\documentclass[11pt]{article}
\usepackage[margin=1in]{geometry}
\usepackage[utf8]{inputenc} 
\usepackage[T1]{fontenc}    
\usepackage[colorlinks=true,allcolors=magenta]{hyperref}       
\usepackage{amsfonts}       
\usepackage{microtype}      
\usepackage{palatino}
\usepackage{mathpazo}
\usepackage[round]{natbib}

\usepackage{amsmath,amsbsy,amsthm,commath,mathrsfs}
\usepackage{algorithm,algorithmic}

\def\ddefloop#1{\ifx\ddefloop#1\else\ddef{#1}\expandafter\ddefloop\fi}
\def\ddef#1{\expandafter\def\csname c#1\endcsname{\ensuremath{\mathcal{#1}}}}
\ddefloop ABCDEFGHIJKLMNOPQRSTUVWXYZ\ddefloop
\def\ddef#1{\expandafter\def\csname v#1\endcsname{\ensuremath{\boldsymbol{#1}}}}
\ddefloop ABCDEFGHIJKLMNOPQRSTUVWXYZabcdefghijklmnopqrstuvwxyz\ddefloop
\def\ddef#1{\expandafter\def\csname v#1\endcsname{\ensuremath{\boldsymbol{\csname #1\endcsname}}}}
\ddefloop {alpha}{beta}{gamma}{delta}{epsilon}{varepsilon}{veps}{zeta}{eta}{theta}{vartheta}{iota}{kappa}{lambda}{mu}{nu}{xi}{pi}{varpi}{rho}{varrho}{sigma}{varsigma}{tau}{upsilon}{phi}{varphi}{chi}{psi}{omega}{Gamma}{Delta}{Theta}{Lambda}{Xi}{Pi}{Sigma}{varSigma}{Sig}{Upsilon}{Phi}{Psi}{Omega}{ell}\ddefloop
\def\ddef#1{\expandafter\def\csname sc#1\endcsname{\ensuremath{\mathscr{#1}}}}
\ddefloop ABCDEFGHIJKLMNOPQRSTUVWXYZ\ddefloop

\newcommand\T{{\ensuremath{\scriptscriptstyle{\top}}}} 
\newcommand\Span{\ensuremath{\operatorname{span}}} 
\newcommand{\R}{\ensuremath{\mathbb{R}}} 
\newcommand\EndProofParagraph{\hfill\qed}

\newtheorem{lemma}{Lemma}
\newtheorem{theorem}{Theorem}

\usepackage[affil-sl]{authblk}
\title{Parameter identification in Markov chain choice models}
\author{Arushi Gupta}
\author{Daniel Hsu}
\affil{Columbia University, New York, NY}

\begin{document}

\maketitle

\begin{abstract}%
  This work studies the parameter identification problem for the Markov chain choice model of Blanchet, Gallego, and Goyal used in assortment planning.  In this model, the product selected by a customer is determined by a Markov chain over the products, where the products in the offered assortment are absorbing states.  The underlying parameters of the model were previously shown to be identifiable from the choice probabilities for the all-products assortment, together with choice probabilities for assortments of all-but-one products.  Obtaining and estimating choice probabilities for such large assortments is not desirable in many settings.  The main result of this work is that the parameters may be identified from assortments of sizes two and three, regardless of the total number of products.  The result is obtained via a simple and efficient parameter recovery algorithm.

\end{abstract}

\section{Introduction}

In assortment planning, the seller's goal is to select a subset of products
(called an \emph{assortment}) to offer to a customer so as to maximize the
expected revenue.
This task can be formulated as an optimization problem given the revenue
generated from selling each product, along with a probabilistic model of the
customer's preferences for the products.
Such a \emph{discrete choice model} must capture the customer's
\emph{substitution behavior} when, for instance, the offered assortment does not
contain the customer's most preferred product.

Our focus in this paper is the Markov chain choice model (MCCM) proposed by
\citet{blanchet2016markov}.
In this model, the product selected by the customer is determined by a Markov
chain over products where the products in the offered assortment are absorbing
states.
The current state represents the desired product; if that product is not
offered, the customer transitions to another product according to the Markov
chain probabilities, and the process continues until the desired product is
offered or the customer leaves.
MCCM generalizes widely-used discrete choice models such as the multinomial
logit model~\citep{luce2005individual,plackett1975analysis}, as well as other
generalized attraction models~\citep{gallego2014general}; it also
well-approximates other random utility models found in the literature such as
mixed multinomial logit models~\citep{mcfadden2000mixed}.
At the same time, the MCCM permits computationally efficient unconstrained
assortment optimization as well as efficient approximation algorithms in the
constrained case~\citep{blanchet2016markov,desir2015capacity}; this stands in
contrast to some richer models such as mixed multinomial logit
models~\citep{rusmevichientong2010assortment} and the nested logit
model~\citep{davis2014assortment} for which assortment optimization is generally
intractable.
This combination of expressiveness and computational tractability makes MCCM
very attractive for use in assortment planning.

A crucial step in this overall enterprise---e.g., before assortment optimization
may take place---is the estimation of the choice model's parameters from
observational data.
Parameter estimation for MCCM is only briefly considered in the original work of
\citet{blanchet2016markov}.
In that work, it is shown that the parameters can be determined from the choice
probabilities for the all-products assortment, together with the assortments
comprised of all-but-one product.
This is not satisfactory because it may be unrealistic or unprofitable to offer
assortments of such large cardinality.
Therefore, it is desirable to be able to determine the parameters from choice
probabilities for smaller cardinality assortments.
We note that this is indeed possible for simpler choice models such as
the multinomial logit model~\citep[see, e.g.,][]{train2009discrete}, but these
simpler models are limited in expressiveness---for example, they cannot express
heterogeneous substitution behavior.

In this paper, we show that the MCCM parameters can be identified from the
choice probabilities for assortments of sizes as small as two and three,
independent of the total number of products.\footnote{%
  We focus on identifiability because estimation of choice probabilities from
  observational data is fairly straightforward, especially when the assortments
  have small cardinality.
  However, this issue is revisited in Section~\ref{sec:discussion} in the context of sample complexity.
}
We also give a simple and efficient algorithm for reconstructing the parameters
from these choice probabilities.

\section{Model and notation}

In this section, we describe the Markov chain choice model (MCCM) of
\citet{blanchet2016markov}, along with notations used for choice probabilities
and model parameters.

The set of $n$ products in the system is denoted by $\cN := \{1,2,\dotsc,n\}$.
The ``no purchase'' option is denoted by product $0$.
Upon offering an assortment $S \subseteq \cN$, the set of possible outcomes is
$S_+ := S \cup \{0\}$: either some product in $S$ is purchased, or no
product is purchased.

Underlying the MCCM is a Markov chain with state space $\cN_+$.
The (true) parameters of the model are the initial state probabilities $\vlambda
= (\lambda_i)_{i \in \cN_+}$ and the transition probabilities $\vrho =
(\rho_{i,j})_{(i,j) \in \cN_+ \times \cN_+}$ (a row stochastic matrix).
The transition probabilities satisfy the following properties:
\begin{enumerate}
  \item
    $\rho_{0,0} = 1$ and $\rho_{0,j} = 0$ for $j \in \cN$ (i.e., the ``no
    purchase'' state is absorbing).

  \item
    $\rho_{i,i} = 0$ for $i \in \cN$ (i.e., no self-loops in product states).

  \item
    The submatrix $\tilde\vrho := (\rho_{i,j})_{(i,j) \in \cN \times \cN}$ is
    irreducible.

\end{enumerate}
We use $\vrho_i = (\rho_{i,j})_{j \in \cN_+}$ to denote the $i$-th row of
$\vrho$.

In MCCM, the customer arrives at a random initial state $X_1$ chosen according
to $\vlambda$.
At time $t = 1,2,\dotsc$:
\begin{itemize}
  \item
    If $X_t = 0$, the customer leaves the system without purchasing a product.

  \item
    If the product $X_t$ is offered (i.e., $X_t \in S$), the customer purchases
    $X_t$ and leaves.

  \item
    If the product $X_t$ is not offered (i.e., $X_t \notin S$), the customer
    transitions to a new random state $X_{t+1}$ chosen according to
    $\vrho_{X_t}$ and the process continues in time step $t+1$ as if the
    customer had initially arrived at $X_{t+1}$.

\end{itemize}
Another way to describe this process is that the Markov chain distribution is
temporarily modified so that the states $S_+$ are absorbing, and the customer
purchases the product upon reaching such a state (or makes no purchase if the
state is $0$).
The irreducibility of $\tilde\vrho$ ensures that the customer eventually leaves
the system (i.e., an absorbing state is reached).
Note that only the identity of the final (absorbing) state is observed, as it
corresponds to either a purchase or non-purchase.
The $(X_t)_{t=1,2,\dotsc,}$ themselves do not correspond to observable customer
behavior, and hence the model parameters $\vlambda$ and $\vrho$ cannot be
directly estimated.

The choice probabilities are denoted by $\pi(j,S)$ for $S \subseteq \cN$ and $j
\in S_+$: this is the probability that $j$ is the final state in the
aforementioned process.
\citet{blanchet2016markov} relate the choice probabilities and the parameters
$\vlambda$ and $\vrho$ as follows:
\begin{align}
  \lambda_j
  & \ = \
  \pi(j,\cN) \,,
  & \rho_{i,j}
  & \ = \
  \begin{cases}
    1
    & \text{if $i = 0$ and $j = 0$} \,, \\
    \displaystyle
    \frac{\pi(j,\cN \setminus \{i\}) - \pi(j,\cN)}{\pi(i,\cN)}
    & \text{if $i \in \cN$, $j \in \cN_+$, and $i \neq j$} \,, \\
    0
    & \text{otherwise} \,.
  \end{cases}
  \label{eq:ident}
\end{align}
The relations in Equation~\eqref{eq:ident} show that the parameters may be
identified from choice probabilities for the assortments $S = \cN$ and $S = \cN
\setminus \{i\}$ for $i \in \cN$.
These choice probabilities may be directly estimated from observations upon
offering such assortments to customers.

\section{Main result}

The following theorem establishes identifiability of the MCCM parameters from
choice probabilities for assortments of sizes as small as two and three.
\begin{theorem}
  \label{thm:main}
  There is an efficient algorithm that, for any $r \in \{2,3,\dotsc,n-1\}$,
  when given as input the choice probabilities $(\pi(j,S))_{j \in S_+}$ for
  all assortments $S \subseteq \cN$ of cardinality $r$ and $r+1$ for a Markov
  chain choice model, returns the parameters $\vlambda$ and $\vrho$ of the
  model.
\end{theorem}

The number of assortments for which the algorithm actually requires choice
probabilities is $O(n^2)$ when $r \leq n/2$, which is far fewer than
$\binom{n}{r} + \binom{n}{r+1}$, the total numbers of assortments of sizes $r$
and $r+1$.
The details of this bound are shown following the proof of
Theorem~\ref{thm:main}.
However, to simplify the presentation, we describe our parameter recovery
algorithm as using choice probabilities for all assortments of sizes $r$ and
$r+1$.

The main steps of our algorithm, shown as Algorithm~\ref{alg:main}, involve
setting up and then solving systems of linear equations that (as we will prove)
determine the unknown parameters $\vlambda$ and $(\vrho_i)_{i \in \cN}$.
(Note that $\vrho_0$ is already known.)
The coefficients of the linear equations are determined by the given choice
probabilities via \emph{conditional choice probabilities} $\pi(j,S \mid i)$ for
$S \subseteq \cN$ and $i,j \in \cN_+$, defined as follows:
\begin{align}
  \pi(j,S \mid i)
  & \ := \
  \Pr\del{
    \text{state $j$ is reached before any state in $S_+ \setminus \{j\}$}
    \mid \text{initial state is $i$}
  }
  \,.
  \label{eq:condprob-defn}
\end{align}

\begin{algorithm}[t]
  \renewcommand\algorithmicrequire{\textbf{input}}
  \renewcommand\algorithmicensure{\textbf{output}}
  \caption{Parameter recovery algorithm for Markov chain choice model}
  \label{alg:main}
  \begin{algorithmic}[1]
    \REQUIRE
    For some $r \in \{2,3,\dotsc,n-1\}$, choice probabilities $(\pi(j,S))_{j \in
    S_+}$ for all assortments $S \subseteq \cN$ of sizes $r$ and $r+1$.

    \ENSURE
    Parameters $\hat\vlambda$ and $\hat\vrho$.

    \FOR{$i \in \cN$}

      \STATE
      Solve the following system of linear equations
      for $\hat\vrho_i = (\hat\rho_{i,k})_{k \in \cN_+}$:
      \begin{align}
        \sum_{k \in \cN_+} \pi(j,S \mid k) \cdot \hat\rho_{i,k}
        & \ = \
        \pi(j,S \mid i)
        \qquad
        \text{for all $S \in \binom{\cN}{r}$ s.t.~$i \notin S$ and $j \in S_+$}
        \,,
        \label{eq:linsys-rho}
      \end{align}
      where $\binom{\cN}{r}$ denotes the family of subsets of $\cN$ of size $r$,
      and $\pi(j,S \mid k)$ is defined in Equation~\eqref{eq:condprob}.

    \ENDFOR

    \STATE
    Solve the following system of linear equations for $\hat\vlambda =
    (\hat\lambda_i)_{i \in \cN_+}$:
    \begin{align}
      \sum_{k \in \cN_+} \pi(j,S \mid k) \cdot \hat\lambda_k
      & \ = \
      \pi(j,S)
      \qquad
      \text{for all $S \in \binom{\cN}{r}$ and $j \in S_+$}
      \label{eq:linsys-lambda}
    \end{align}

    \RETURN
    $\hat\vlambda$ and $\hat\vrho$.

  \end{algorithmic}
\end{algorithm}

Note that the initial state in the MCCM is not observed, so these conditional
probabilities cannot be directly estimated.
Nevertheless, they can be indirectly estimated via the following relationship
between the conditional choice probabilities and the (unconditional) choice
probabilities.
\begin{lemma}
  \label{lem:ident}
  For any $S \subseteq \cN$ and $i,j \in S_+$,
  \begin{align}
    \pi(j,S \mid i)
    & \ = \
    \begin{cases}
      1
      & \text{if $i = j$} \,, \\
      \displaystyle
      \frac{\pi(j,S) - \pi(j,S \cup \{i\})}{\pi(i,S \cup \{i\})}
      & \text{if $i \in \cN \setminus S$} \,, \\
      0
      & \text{if $i \in S_+ \setminus \{j\}$} \,.
    \end{cases}
    \label{eq:condprob}
  \end{align}
\end{lemma}
\begin{proof}
  The cases where $i = j$ ($\Rightarrow \pi(j,S \mid i) = 1$) and $i \in S_+
  \setminus \{j\}$ ($\Rightarrow \pi(j,S \mid i) = 0$) are clear from the
  definition in Equation~\eqref{eq:condprob-defn}.
  It remains to handle the case where $i \in \cN \setminus S$.
  Fix such a product $i$, and observe that
  \begin{align*}
    \pi(j,S)
    & \ = \
    \Pr\del{
      \text{$j$ is reached before $S_+ \setminus \{j\}$}
    }
    \\
    & \ = \
    \Pr\del{
      \text{$j$ is reached before $S_+ \setminus \{j\}$}
      \wedge
      \text{$i$ is not reached before $S_+$}
    }
    \\
    & \qquad
    +
    \Pr\del{
      \text{$j$ is reached before $S_+ \setminus \{j\}$}
      \wedge
      \text{$i$ is reached before $S_+$}
    }
    \\
    & \ = \
    \Pr\del{
      \text{$j$ is reached before $(S_+ \cup \{i\}) \setminus \{j\}$}
    }
    \\
    & \qquad
    +
    \Pr\del{
      \text{$j$ is reached before $S_+ \setminus \{j\}$}
      \mid
      \text{$i$ is reached before $S_+$}
    }
    \\
    & \qquad\quad
    \cdot
    \Pr\del{
      \text{$i$ is reached before $S_+$}
    }
    \\
    & \ = \
    \Pr\del{
      \text{$j$ is reached before $(S \cup \{i\})_+ \setminus \{j\}$}
    }
    \\
    & \qquad
    +
    \Pr\del{
      \text{$j$ is reached before $S_+ \setminus \{j\}$}
      \mid
      \text{initial state is $i$}
    }
    \\
    & \qquad\quad
    \cdot
    \Pr\del{
      \text{$i$ is reached before $(S \cup \{i\})_+ \setminus \{i\}$}
    }
    \\
    & \ = \
    \pi(j,S \cup \{i\}) + \pi(j,S \mid i) \cdot \pi(i,S)
    \,.
  \end{align*}
  The penultimate step uses the Markov property and the case condition that $i
  \in \cN \setminus S$.
  Rearranging the equation gives the relation claimed by the lemma in this case.
\end{proof}

Lemma~\ref{lem:ident} shows that the conditional choice probabilities for
assortments $S$ of size $r$ can be determined from the unconditional choice
probabilities of assortments of size $r$ and $r+1$.
The systems of linear equations used in Algorithm~\ref{alg:main}
(Equations~\eqref{eq:linsys-rho} and~\eqref{eq:linsys-lambda}) are defined in
terms of these conditional choice probabilities and hence are ultimately defined
in terms of the unconditional choice probabilities provided as input to
Algorithm~\ref{alg:main}.

It is clear that the true MCCM parameters $\vlambda$ and $\vrho$ satisfy the
systems of linear equations in Equations~\eqref{eq:linsys-rho}
and~\eqref{eq:linsys-lambda}.
However, what needs to be proved is that they are uniquely determined by these
linear equations; this is the main content of the proof of
Theorem~\ref{thm:main}.

\section{Proof of Theorem~\ref{thm:main}}

In this section, we give the proof of Theorem~\ref{thm:main}.

\subsection{The case without the ``no purchase'' option}

For sake of clarity, we first give the proof in the case where the ``no
purchase'' option is absent.
This can be regarded as the special case where $\lambda_0 = 0$ and $\rho_{i,0} =
0$ for all $i \in \cN$.
So here we just regard $\vlambda = (\lambda_j)_{j \in \cN}$ and each $\vrho_i =
(\rho_{i,j})_{j \in \cN}$ as probability distributions on $\cN$.
The general case will easily follow from the same arguments with minor
modification.

\subsubsection{Proof strategy}

We make use of the following result about M-matrices, i.e., the class of
matrices $\vA$ that can be expressed as $\vA = s \vI - \vB$ for some $s>0$ and
non-negative matrix $\vB$ with spectral radius at most $s$.
(Here, $\vI$ denotes the identity matrix of appropriate dimensions.)
In particular, the matrix $\vI - \vrho$ is a (singular) M-matrix that is also
irreducible.
\begin{lemma}[See, e.g., Theorems 6.2.3 \& 6.4.16
  in~\citealp{berman1994nonnegative}]
  \label{lem:m-matrix}
  If $\vA \in \R^{p \times p}$ is an irreducible M-matrix (possibly singular),
  then every principal submatrix\footnote{%
    Recall that a \emph{principal submatrix} of a $p \times p$ matrix $\vA$ is a submatrix obtained by removing from $\vA$ the rows and columns indexed by some set $I \subseteq [p]$.%
    } of $\vA$, other than $\vA$ itself, is
  non-singular.
  If $\vA$ is also singular, then it has rank $p-1$.
\end{lemma}

For each $S \in \binom{\cN}{r}$ and $j \in S$, define the vector
\begin{align*}
  \vh_{j,S}
  & \ := \
  (\pi(j,S \mid k))_{k \in \cN}
  \,.
\end{align*}
For each $i \in \cN$, the collection of the vectors $\{ \vh_{j,S} : S \not\ni i
\wedge j \in S \}$ provide the left-hand side coefficients in
Equation~\eqref{eq:linsys-rho} for $\hat\vrho_i$.
We'll show that the span of these vectors (in fact, a particular subset of them)
has dimension at least $n-1$.
This is sufficient to conclude that $\vrho_i$ is the unique solution to the
system of equations in Equation~\eqref{eq:linsys-rho} because it has at most
$n-1$ unknown variables, and it is clear that $\vrho_i$ satisfies the system of
equations.
(In fact, there are really only $n-2$ unknown variables, because we can force
$\hat\rho_{i,i} = 0$ and $\hat\rho_{i,n} = 1 - \sum_{k=1}^{n-1}
\hat\rho_{i,k}$.)
For the same reason, it is also sufficient to conclude that $\vlambda$ is the
unique solution to the system of equations in Equation~\eqref{eq:linsys-lambda}
(where, in fact, we may use all vectors $\{ \vh_{j,S} : S \in \binom{\cN}{r}
\wedge j \in S \}$).

\subsubsection{Rank of linear equations from a single assortment}

We begin by characterizing the space spanned by $\{ \vh_{j,S} : j \in S \}$ for
a fixed $S \in \binom{\cN}{r}$.
We claim, by Lemma~\ref{lem:ident}, that the vectors in $\{ \vh_{j,S} : j \in S
\}$ are linearly independent.
Indeed, if this collection of vectors is arranged in a matrix $[ \vh_{j,S} : j
\in S ]$, then the submatrix obtained by selecting rows corresponding to $j \in
S$ is the $S \times S$ identity matrix.
Thus we have proved
\begin{lemma}
  \label{lem:dim}
  For any $S \in \binom{\cN}{r}$,
  $\dim\del{\Span\{ \vh_{j,S} : j \in S \}} = |S| = r$.
\end{lemma}

Note that in the case $r = n-1$, we are done.
But when $r < n-1$, the linear equations given by the $\{ \vh_{j,S} : j \in S
\}$ may not uniquely determine the $\vrho_i$ for $i \in \cN \setminus S$.
To overcome this, we need to be able to combine linear equations derived from
multiple assortments.
However, for a sum of subspaces $V$ and $W$,
\begin{align*}
  \dim (V + W)
  & \ \neq \
  \dim(V) + \dim(W)
\end{align*}
unless $V$ and $W$ are orthogonal.
In our case, the subspaces $\Span\{ \vh_{j,S} : j \in S \}$ and $\Span\{
  \vh_{j,S'} : j \in S' \}$ for different assortments $S$ and $S'$ are
\emph{not} necessarily orthogonal (even if $S$ and $S'$ are disjoint).
So a different argument is needed.

\subsubsection{Rank of linear equations from multiple assortments}

Our aim is to show that the intersection of subspaces $V := \Span\{ \vh_{j,S} :
j \in S \} \cap \Span\{ \vh_{j,S'} : j \in S' \}$ for different assortments $S$
and $S'$ cannot have high dimension.
We do this by showing that the intersection is orthogonal to a subspace of high
dimension.

For each $i \in \cN$, let $\va_i$ denote the $i$-th row of the matrix $\vA :=
\vI - \vrho$ (which is an M-matrix).
That is, $\va_i := \ve_i - \vrho_i$, where $\ve_i$ is the $i$-th coordinate
basis vector.
Recall that if $i \in \cN \setminus S$, then $\vrho_i$ satisfies
Equation~\eqref{eq:linsys-rho}.
This fact can be written in our new notation as
\begin{align*}
  \vh_{j,S}^\T\ve_i - \vh_{j,s}^\T\vrho_i
  & \ = \
  \va_i^\T\vh_{j,S}
  \ = \ 0
  \,,
  \quad j \in S
  \,.
\end{align*}
In other words,
\begin{lemma}
  \label{lem:perp}
  For any $S \in \binom{\cN}{r}$,
  $\Span\{ \vh_{j,S} : j \in S \} \perp \Span\{ \va_i : i \in \cN \setminus S
  \}$.
\end{lemma}

Now consider two assortments $S$ and $S'$, and the intersection of their
respective subspaces.
It follows from Lemma~\ref{lem:perp} that
\begin{align*}
  \Span\{ \vh_{j,S} : j \in S \} \cap \Span\{ \vh_{j,S'} : j \in S' \}
  & \ \perp \
  \Span\{ \va_i : i \in \cN \setminus (S \cap S') \}
  \,.
\end{align*}
This orthogonality is the key to lower-bounding the dimension of the sum of
these subspaces, which we capture in the following general lemma.
\begin{lemma}
  \label{lem:dim-multiple}
  Let $\scS$ be a family of subsets of $\cN$, $S'$ be a subset of $\cN$,
  and $\scS' := \scS \cup \{ S' \}$.
  Define the subspaces
  \begin{align*}
    V_{\scS} & \ := \ \Span\{ \vh_{j,S} : S \in \scS , j \in S \} \,, \\
    V_{S'} & \ := \ \Span\{ \vh_{j,S'} : j \in S' \} \,, \\
    V_{\scS'} & \ := \ V_{\scS} + V_{S'} \,.
  \end{align*}
  Then
  \begin{align*}
    \dim\del{ V_{\scS'} }
    & \ \geq \
    \dim\del{ V_{\scS} }
    + |S'|
    - \max\cbr{ 1 ,\, \abs{\del{\textstyle\bigcup_{S \in \scS} S} \cap S'} }
    \,.
  \end{align*}
\end{lemma}
\begin{proof}
  Let $S_0 := \bigcup_{S \in \scS} S$.
  Fix any $\vv \in V_{\scS} \cap V_{S'}$.
  Then, by Lemma~\ref{lem:perp}, $\va_i^\T\vv = 0$ for all $i \in (\cN \setminus
  S_0) \cup (\cN \setminus S') = \cN \setminus (S_0 \cap S')$.
  In other words,
  \begin{align*}
    V_{\scS} \cap V_{S'}
    & \ \perp \
    W
    \,,
  \end{align*}
  where $W := \Span\{ \va_i : i \in \cN \setminus (S_0 \cap S') \}$, and
  \begin{align*}
    \dim\del{V_{\scS} \cap V_{S'}}
    & \ \leq \
    \dim(W^\perp)
    \ = \
    n - \dim(W)
    \,.
  \end{align*}
  To determine $\dim(W)$, observe that $W$ is the span of rows of certain rows
  of the M-matrix $\vA$.
  By Lemma~\ref{lem:m-matrix}, the principal submatrix of $\vA$ corresponding to
  $\cN \setminus (S_0 \cap S')$ is either non-singular (when $S_0 \cap S' \neq
  \emptyset$) or is $\vA$ itself; in either case, it has rank $n - \max\{ 1,
  |S_0 \cap S'| \}$.
  Hence,
  \begin{align*}
    \dim(W)
    & \ = \
    n - \max\{ 1, |S_0 \cap S'| \}
  \end{align*}
  as well.
  Combining the dimension formula with the last two equation displays gives
  \begin{align*}
    \dim(V_{\scS'})
    & \ = \
    \dim(V_{\scS} + V_{S'})
    \\
    & \ = \
    \dim(V_{\scS}) + \dim(V_{S'}) - \dim(V_{\scS'} \cap V_{S'})
    \\
    & \ \geq \
    \dim(V_{\scS}) + \dim(V_{S'}) - \max\{ 1 ,\, |S_0 \cap S'| \}
    \,.
  \end{align*}
  The claim now follows from Lemma~\ref{lem:dim}.
\end{proof}

\subsubsection{Choice of assortments}
\label{sec:choice-of-assortments}

We now choose a collection of assortments and argue, via Lemma~\ref{lem:dim} and
Lemma~\ref{lem:dim-multiple}, that they define linear equations of sufficiently
high rank.
Specifically, for each $i \in \cN$, we need a collection $\scS \subset
\binom{\cN}{r}$ such that each $S \in \scS$ does not contain $i$, and
\begin{align}
  \dim\del{
    \Span\bigcup_{S \in \scS} \{ \vh_{j,S} : j \in S \}
  }
  & \ \geq \
  n-1
  \,.
  \label{eq:dim-bound}
\end{align}

\begin{lemma}
  \label{lem:dim-common-intersect}
  Suppose the assortments $S_1, S_2, \dotsc, S_T \in \binom{\cN}{r}$ have a
  pairwise common intersection $S_{\cap} = S_t \cap S_{t'}$ for all $t \neq t'$,
  and $|S_{\cap}| = r-1$.
  Then $\dim\del[0]{
    \Span\bigcup_{t=1}^T \{ \vh_{j,S_t} : j \in S_t \}
  } \geq T + r - 1$.
\end{lemma}
\begin{proof}
  Let $d_\tau := \dim\del[0]{ \Span\bigcup_{t=1}^{\tau} \{ \vh_{j,S_t} : j \in
  S_t \} }$ for $\tau \in \{1,2,\dotsc,T\}$.
  By Lemma~\ref{lem:dim}, we know that $d_1 = r$.
  Now, assume $d_\tau \geq \tau + r - 1$, and use the fact $r\geq2$ and
  Lemma~\ref{lem:dim-multiple} to conclude that $d_{\tau+1} \geq d_\tau + r -
  (r-1) = d_\tau + 1 \geq \tau + r$.
  The claim now follows by induction.
\end{proof}

Fix any $i \in \cN$ and $S_\cap \in \binom{\cN\setminus\{i\}}{r-1}$, and observe
that $|\cN \setminus (S_\cap \cup \{i\})| = n - r$.
Consider the collection of size-$r$ assortments given by
\begin{equation}
  \scS
  \ := \
  \cbr{
    S_\cap \cup \{ k \} : k \in \cN \setminus (S_\cap \cup \{i\})
  }
  \,.
  \label{eq:assortment-choice}
\end{equation}
These assortments do not contain $i$, they have the common intersection
$S_\cap$, with $|S_\cap| = r-1$, and there are $n-r$ assortments in total.
So by Lemma~\ref{lem:dim-common-intersect}, the collection $\scS$ satisfies the
dimension bound in Equation~\eqref{eq:dim-bound}.

As was already argued in the proof strategy, this suffices to establish the
uniqueness of the $\vrho_i$ and $\vlambda$ as solutions to the respective
systems of linear equations in Equation~\eqref{eq:linsys-rho} and
Equation~\eqref{eq:linsys-lambda}.

This concludes the proof of Theorem~\ref{thm:main} without the ``no purchase''
option.
\EndProofParagraph

\subsection{The general case with the ``no purchase'' option}

We now consider the general case, where the ``no purchase'' option is present.
The main difference relative to the previous subsection is that $\vrho$ is no
longer irreducible, as the ``no purchase'' state $0$ is absorbing.
However, the submatrix $\tilde\vrho = (\rho_{i,j})_{(i,j) \in \cN \times \cN}$
\emph{is} irreducible, so $\vI - \tilde\vrho$ is an irreducible M-matrix.

The definition of $\vh_{j,S}$, for $S \subseteq \cN$ and $j \in S_+$, is now
taken to be
\begin{align*}
  \vh_{j,S}
  & \ := \
  (\pi(j,S \mid k))_{k \in \cN_+}
  \,.
\end{align*}
Because the indexing starts at $0$, we still define $\va_i$ to be the $i$-th row
of $\vA = \vI - \vrho$, so $\va_i = \ve_i - \vrho_i$.
(In particular, $\va_0$ is the all-zeros vector.)

With these definitions, we have the following analogue of Lemma~\ref{lem:dim}
and Lemma~\ref{lem:perp}:
\begin{lemma}
  \label{lem:dimperp}
  For any $S \in \binom{\cN}{r}$,
  \begin{align*}
    \dim(\Span\{\vh_{j,S} : j \in S_+\})
    & \ = \ |S_+| \ = \ r+1 \,,
    \\
    \Span\{ \vh_{j,S} : j \in S_+ \}
    & \ \perp \ \Span\{ \va_i : i \in \cN \setminus S \}
    \,.
  \end{align*}
\end{lemma}
Here, the key difference is that the dimension is $r+1$, rather than just $r$.

We now establish an analogue of Lemma~\ref{lem:dim-multiple} (which is
typographically nearly identical).
\begin{lemma}
  \label{lem:dim-multipler}
  Let $\scS$ be a family of subsets of $\cN$, $S'$ be a subset of $\cN$,
  and $\scS' := \scS \cup \{ S' \}$.
  Define the subspaces
  \begin{align*}
    V_{\scS} & \ := \ \Span\{ \vh_{j,S} : S \in \scS , j \in S_+ \} \,, \\
    V_{S'} & \ := \ \Span\{ \vh_{j,S'} : j \in S_+' \} \,, \\
    V_{\scS'} & \ := \ V_{\scS} + V_{S'} \,.
  \end{align*}
  Then
  \begin{align*}
    \dim\del{ V_{\scS'} }
    & \ \geq \
    \dim\del{ V_{\scS} }
    + |S'|
    - \max\cbr{ 1 ,\, \abs{\del{\textstyle\bigcup_{S \in \scS} S} \cap S'} }
    \,.
  \end{align*}
\end{lemma}
\begin{proof}
  The proof is nearly the same as that of Lemma~\ref{lem:dim-multiple}.
  Define $S_0 := \bigcup_{S \in \scS} S$ and take $\vv \in V_{\scS} \cap
  V_{S'}$.
  By Lemma~\ref{lem:dimperp}, $V_{\scS} \cap V_{S'} \perp W$, where $W :=
  \Span\{ \va_i : i \in \cN \setminus (S_0 \cap S') \}$, and
  \begin{align*}
    \dim(V_{\scS} \cap V_{S'})
    & \ \leq \ \dim(W^\perp) \ = \ n+1 - \dim(W)
    \,.
  \end{align*}
  We now use the fact that $\vI - \tilde\vrho$, which is a submatrix of $\vA$,
  is an irreducible M-matrix.
  By Lemma~\ref{lem:m-matrix}, the principal submatrix of $\vA$ corresponding to
  $\cN \setminus (S_0 \cap S')$ is either non-singular (when $S_0 \cap S' \neq
  \emptyset$) or is $\vI - \tilde\vrho$; in either case, it has rank $n - \max\{
    1, |S_0 \cap S'| \}$.
  So we have
  \begin{align*}
    \dim(W)
    & \ = \
    n - \max\{ 1, |S_0 \cap S'| \}
    \quad
    \text{and}
    \quad
    \dim(W^\perp)
    \ = \
    1 + \max\{ 1, |S_0 \cap S'| \}
    \,.
  \end{align*}
  Finishing the proof as in Lemma~\ref{lem:dim-multiple}, we have
  \begin{align*}
    \dim(V_{\scS'})
    & \ \geq \
    \dim(V_{\scS}) + \dim(V_{S'}) - 1 - \max\{ 1 ,\, |S_0 \cap S'| \}
    \\
    & \ \geq \
    \dim(V_{\scS}) + r - \max\{ 1 ,\, |S_0 \cap S'| \}
  \end{align*}
  where the second inequality uses Lemma~\ref{lem:dimperp} (instead of
  Lemma~\ref{lem:dim}).
\end{proof}

The choice of assortments demonstrating the subspace of required dimension is
the same as before, except now we show that the dimension is at least $n$.
Again, fix some $i \in \cN$, and choose the collection of $n-r$ assortments
$\scS \subseteq \binom{\cN}{r}$ as before (described in and directly before
Equation~\eqref{eq:assortment-choice}).
Following the inductive argument in the proof of
Lemma~\ref{lem:dim-common-intersect}, but now using Lemma~\ref{lem:dimperp} and
Lemma~\ref{lem:dim-multipler} (instead of Lemma~\ref{lem:dim} and
Lemma~\ref{lem:dim-multiple}), we have
\begin{align*}
  \dim\del{
    \Span\bigcup_{S \in \scS} \{ \vh_{j,S} : j \in S \}
  }
  & \ \geq \
  (r + 1) + (n-r-1) \cdot (r - \max\{ 1, r-1 \})
  \ = \
  n
  \,.
\end{align*}
Since each of the systems of linear equations from
Equation~\eqref{eq:linsys-rho} and Equation~\eqref{eq:linsys-lambda} have (at
most) $n$ unknown variables, we conclude that the $\vrho_i$ and $\vlambda$ are
unique as solutions to their respective systems of linear equations.

This concludes the proof of Theorem~\ref{thm:main}.
\EndProofParagraph

\subsection{Total number of assortments required}

We now show that the number of assortments for which we need the choice
probabilities is $O(n^2)$ for $r \leq n/2$.
Indeed, the construction given above based on
Lemma~\ref{lem:dim-common-intersect} can be used to avoid using all assortments
of size $r$ (and $r+1$) in Algorithm~\ref{alg:main}.

We choose two sets $S_\cap, S_\cap' \in \binom{\cN}{r-1}$, which shall serve as
``common intersection sets'' (in the sense used in
Section~\ref{sec:choice-of-assortments}), as follows.
The first set $S_\cap \in \binom{\cN}{r-1}$ is chosen arbitrarily; it serves as
the common intersection set for all $i \in \cN \setminus S_\cap$.
The second set $S_\cap' \in \binom{\cN \setminus S_\cap}{r-1}$ is chosen
arbitrarily as long as it is disjoint from $S_\cap$ (which is possible because
$r \leq n/2$); it serves as the common intersection set for $i \in S_\cap$.

For each $i \in \cN \setminus S_\cap$, we need the equations for the
assortments $S_\cap \cup \{k\}$ for all $k \in \cN \setminus (S_\cap \cup
\{i\})$.
Obtaining the equations for one such $S_\cap \cup \{k\}$ requires choice
probabilities for assortments $S_\cap \cup \{k\}$ and $S_\cap \cup \{k\} \cup
\{j\}$ for $j \in \cN \setminus (S_\cap \cup \{k\})$ as per Lemma~\ref{lem:ident}.
In total, for all $i \in \cN \setminus S_\cap$, we need choice probabilities for
$O(n^2)$ assortments.
For the remaining $i \in S_\cap$, we use the same argument for the disjoint
common intersection set $S_\cap'$, and thus require the choice probabilities for
at most another $O(n^2)$ assortments.

\section{Discussion}
\label{sec:discussion}

Our main result establishes the identifiability of MCCM parameters from choice
probabilities for assortments of sizes different from $n-1$ and $n$.
This is important because real systems often have cardinality constraints on the
assortment sizes.
While such constraints are typically considered in the context of assortment
optimization~\citep[see, e.g.,][]{desir2015capacity}, it is also important in
the context of parameter estimation.

One complication of using small size assortments to estimate the MCCM parameters
is that the number of different assortments required may be as large as
$O(n^2)$.
In contrast, only $n+1$ assortments are needed when the sizes are $n-1$ and $n$.
On the other hand, the statistical difficulty of estimating choice probabilities
for large assortments may be higher than the same task for smaller assortments.
So the possible trade-offs in sample complexity is not straightforward from
this analysis.
This is an interesting question that we leave to future work.

\section*{Acknowledgments}

We are grateful to Shipra Agrawal and Vineet Goyal for helpful discussions, and
to Vineet for originally suggesting this problem.
This work was supported in part by NSF awards DMR-1534910 and IIS-1563785, a
Bloomberg Data Science Research Grant, a Sloan Research Fellowship, and the
Research Opportunities and Approaches to Data Science (ROADS) grant from the
Data Science Institute at Columbia University.

\bibliographystyle{plainnat}
\bibliography{paper}

\end{document}